\documentclass[10pt]{article}

\usepackage{amsmath,amsbsy,amssymb,latexsym,amsthm,dsfont}
\usepackage{a4wide}
\usepackage{graphics}
\usepackage{graphicx}

\newtheorem{thm}{Theorem}

\newtheorem{Definition}{Definition}
\newtheorem{Remark}{Remark}

\newtheorem{ex}{Example}

\begin{document}

\title{Fractional Variational Principle of Herglotz}

\author{Ricardo Almeida$^1$\\
\texttt{ricardo.almeida@ua.pt}
\and Agnieszka B. Malinowska$^3$\\
\texttt{a.malinowska@pb.edu.pl}}

\date{$^1$Department of Mathematics, University of Aveiro, 3810-193 Aveiro, Portugal\\[0.3cm]
$^3$Faculty of Computer Science, Bia{\l}ystok University of Technology,\\
15-351 Bia\l ystok, Poland}

\maketitle

\begin{abstract}

The aim of this paper is to bring together two approaches to non-conservative systems -- the generalized variational principle of Herglotz and the fractional calculus of variations. Namely, we consider functionals whose extrema are sought, by differential equations that involve Caputo fractional derivatives. The Euler--Lagrange equations are obtained for the fractional variational problems of Herglotz-type and the transversality conditions are derived. The fractional Noether-type theorem for conservative and non-conservative physical systems is proved.


\smallskip

\end{abstract}


\section{Introduction}

As it was pointed out by Cornelius L\'anczos \cite{Lanczos}, frictional and non-conservative forces
are beyond the usual macroscopic variational treatment and, consequently,
beyond the most advanced methods of classical mechanics. Over the years, a number of methods have been presented to circumvent the discrimination against non-conservative systems. We can mention the Rayleigh dissipation function \cite{gol}--which is the best known, the Bateman--
Caldirola--Kanai (BCK) Lagrangian \cite{men}, the generalized variational principle of Herglotz \cite{gue,Herglotz} or the fractional calculus of variations \cite{book:MT,rew}. The aim of this paper is to marge the generalized variational principle of Herglotz with the fractional calculus of variations. In other words, we consider functionals whose extrema are sought, by differential equations that involve Caputo fractional derivatives.

The generalized variational principle was proposed by Gustav Herglotz in 1930 (see \cite{Herglotz}). It generalizes the classical
variational principle by defining the functional  through a differential equation. It reduces to the classical variational integral under
appropriate conditions. The generalized variational principle gives a
variational description of non-conservative processes. For a system, conservative or non-conservative, which can be described
with the generalized variational principle, one can systematically derive conserved
quantities, as shown in \cite{geo2,geo3,geo4}, by applying the Noether-type theorem.

The fractional calculus of variations generalizes the classical variational calculus by considering fractional derivatives into the variational
integrals to be extremized. Fred Riewe \cite{rew} showed that a Lagrangian involving fractional
time derivatives leads to an equation of motion with non-conservative forces such
as friction.  After the seminal papers of Riewe, several different approaches have been developed to generalize the least action principle and the
Euler--Lagrange equations to include fractional derivatives. It has been proved, using the notion of Euler--Lagrange fractional extremal, a Noether-type theorem that combines conservative and non-conservative cases of dynamical systems (see \cite{Gastão+Torres 2007} or \cite{book:MT} and references given there).

The significance of both approaches to non-conservative systems -- the generalized variational principle of Herglotz and the fractional calculus of variations -- motivated this work. The paper is organized as follows. In Section 2, for the reader's convenience, we review the necessary notions of the fractional calculus. Our results are given in next sections: in Section 3 we derive
the Euler-–Lagrange equations and the transversality conditions for fractional variational problems of Herglotz-type with one independent variable (Theorems~\ref{TNC} and \ref{TNC2}), with higher-order fractional derivatives (Theorems \ref{TNC5} and \ref{TNC6}), and for problems with several independent variables (Theorem~\ref{TNC7}). An heuristic method for solving the Euler--Lagrange equations for the fractional Herglotz-type problem is proposed. Finally, in Section 4 we obtain the fractional Noether-type theorem (Theorem~\ref{Noether:without:time} and Theorem~\ref{Noether:without:time:m}) that can be used for conservative and non-conservative physical systems. Throughout the paper we illustrate the new results by specific examples.

\section{Preliminaries}
\label{sec2}

For the convenience of the reader, we present the definitions of fractional operators that will be used in the sequel. For
more on the theory of fractional calculus we refer to \cite{kilbas,Podlubny,Samko}; and for general results on the
fractional calculus of variations to \cite{book:MT} and references therein. \\

Let $x:[a,b]\rightarrow\mathbb{R}$ be a function, $\alpha$ be a positive real number and $n=[\alpha]+1$, where $[\alpha]$ denotes the integer part of $\alpha$. In what follows, we assume that $x$ satisfies appropriate conditions in order to the fractional operators to be well defined. The left and right Riemann--Liouville fractional integrals of order $\alpha$ are given by
$${_aI_t^\alpha}x(t)=\frac{1}{\Gamma(\alpha)}\int_a^t (t-\tau)^{\alpha-1}x(\tau)d\tau,\quad t>a$$
and
$${_tI_b^\alpha}x(x)=\frac{1}{\Gamma(\alpha)}\int_t^b(\tau-t)^{\alpha-1} x(\tau)d\tau,\quad t<b,$$
respectively, where $\Gamma$ represents the Gamma function:
$$\Gamma(z)=\int_0^\infty t^{z-1}e^{-t}\, dt.$$ The left and right Riemann--Liouville fractional derivatives are given by
$${_aD_t^\alpha}x(t) =\frac{1}{\Gamma(n-\alpha)}\frac{d^n}{dt^n}\int_a^t (t-\tau)^{n-\alpha-1}x(\tau) d\tau,\quad t>a$$
and
$${_tD_b^\alpha}x(t)=\frac{(-1)^n}{\Gamma(n-\alpha)}\frac{d^n}{dt^n}\int_t^b (\tau-t)^{n-\alpha-1}x(\tau) d\tau, \quad t<b,$$
respectively. The left and right Caputo fractional derivatives are defined by
$${^C_aD_t^\alpha}x(t) =\frac{1}{\Gamma(n-\alpha)}\int_a^t (t-\tau)^{n-\alpha-1}x^{(n)}(\tau) d\tau,\quad t>a$$
and
$${^C_tD_b^\alpha}x(t)=\frac{(-1)^n}{\Gamma(n-\alpha)}\int_t^b (\tau-t)^{n-\alpha-1}x^{(n)}(\tau) d\tau, \quad t<b,$$
respectively. The integration by parts formula plays a crucial role in deriving the Euler--Lagrange equation. For the left Caputo fractional derivative, this formula is formulated the following way.

\begin{thm}\label{IP} (\cite{book:Klimek}))
Let $\alpha>0$, and $x,y:[a,b]\to\mathbb{R}$ be two functions of class $C^n$, with $n=[\alpha]+1$. Then,
$$\int_{a}^{b}y(t)\cdot {_a^C D_t^\alpha}x(t)dt=\int_a^b x(t)\cdot {_t D_b^\alpha} y(t)dt+\sum_{j=0}^{n-1}\left[{_tD_b^{\alpha+j-n}}y(t) \cdot (-1)^{n-1-j}x^{(n-1-j)}(t)\right]_a^b.$$
\end{thm}

Partial fractional integrals and derivatives for functions of $p$ independent variables are defined in a similar way as is done for integer order derivatives. Let $x:\prod_{i=1}^p [a_i,b_i]\to\mathbb R$ be a function, $\alpha_i$, $i=1,\ldots,p$ be real numbers and define $n_i=[\alpha_i]+1$. Then, we define the partial Riemann--Liouville fractional integrals of order $\alpha_i$ with respect to $t_i$ by
\begin{equation*}
{_{a_{i}}I_{t_{i}}^{{\alpha}_i}}x(t_1,\ldots,t_p)=\frac{1}{\Gamma(\alpha_i)}\int_{a_i}^{t_i}(t_i-\tau_i)^{\alpha_i-1}x(t_1,\ldots,t_{i-1},\tau_i,t_{i+1},
\ldots,t_p)d\tau_i,
\end{equation*}
\begin{equation*}
{_{t_{i}}I_{b_{i}}^{{\alpha}_i}}x(t_1,\ldots,t_p)=\frac{1}{\Gamma(\alpha_i)}\int_{t_i}^{b_i}(\tau_i-t_i)^{\alpha_i-1}x(t_1,\ldots,t_{i-1},\tau_i,t_{i+1},
\ldots,t_p)d\tau_i.
\end{equation*}

Partial Riemann--Liouville and Caputo derivatives are defined by
\begin{equation*}
{_{a_{i}}D_{t_{i}}^{{\alpha}_i}}x(t_1,\ldots,t_p)=\frac{1}{\Gamma(n_i-\alpha_i)}\frac{\partial^{n_i}}{\partial t_i^{n_i}}\int_{a_i}^{t_i}(t_i-\tau_i)^{n_i-\alpha_i-1}x(t_1,\ldots,t_{i-1},\tau_i,t_{i+1},\ldots,t_p)d\tau_i,
\end{equation*}
\begin{equation*}
{_{t_{i}}D_{b_{i}}^{{\alpha}_i}}x(t_1,\ldots,t_p)=\frac{(-1)^{n_i}}{\Gamma(n_i-\alpha_i)}\frac{\partial^{n_i}}{\partial t_i^{n_i}}\int_{t_i}^{b_i}(\tau_i-t_i)^{n_i-\alpha_i-1}x(t_1,\ldots,t_{i-1},\tau_i,t_{i+1},\ldots,t_p)d\tau_i,
\end{equation*}
\begin{equation*}
{^C_{a_{i}}D_{t_{i}}^{\alpha_i}}x(t_1,\ldots,t_p)=\frac{1}{\Gamma(n_i-\alpha_i)}\int_{a_i}^{t_i}(t_i-\tau_i)^{n_i-\alpha_i-1}
\frac{\partial^{n_i}}{\partial t_i^{n_i}}x(t_1,\ldots,t_{i-1},\tau_i,t_{i+1},\ldots,t_p)d\tau_i,
\end{equation*}
\begin{equation*}
{^C_{t_{i}}D_{b_{i}}^{\alpha_i}}x(t_1,\ldots,t_p)=\frac{(-1)^{n_i}}{\Gamma(n_i-\alpha_i)}\int_{t_i}^{b_i}(\tau_i-t_i)^{n_i-\alpha_i-1}
\frac{\partial^{n_i}}{\partial t_i^{n_i}}x(t_1,\ldots,t_{i-1},\tau_i,t_{i+1},\ldots,t_p)d\tau_i.
\end{equation*}


\section{The generalized fractional Euler--Lagrange equations}
\label{sec3}

Consider the differential equation with dependence on Caputo fractional derivatives
\begin{equation}\label{MainProblem}
\frac{dz}{dt}=L(t,x(t),{^C_aD^\alpha_t}x(t),z(t)), \quad  t\in[a,b],
\end{equation}
with the initial condition $z(a)=z_a$, where $t$ is the independent variable, $x=(x_1,\ldots,x_n)$ the dependent variable and $${^C_aD^\alpha_t}x(t):=({^C_aD^{\alpha_1}_t}x_1(t),\ldots,{^C_aD^{\alpha_n}_t}x_n(t)).$$
In what follows we assume that:
\begin{enumerate}
\item $x(a)=x_a$, $x(b)=x_b$, $x_a,x_b\in\mathbb R^n$;
\item  $\alpha_j\in(0,1)$, $j=1,\ldots,n$;
\item  $x\in C^1([a,b],\mathbb R^n)$,  ${^C_aD^\alpha_t}x\in C^1([a,b],\mathbb R^n)$;
\item  the Lagrangian $L:[a,b]\times\mathbb R^{2n+1}\to\mathbb R$ is of class $C^1$ and the maps $t\mapsto \displaystyle\lambda(t)\frac{\partial L}{\partial {^C_aD^{\alpha_j}_t}x_j}[x,z](t)$ are such that $\displaystyle{_tD^{\alpha_j}_b\left(\lambda(t)\frac{\partial L}{\partial {^C_aD^{\alpha_j}_t}x_j}[x,z](t)\right)}$, $j=1,\ldots,n$, exist and are continuous on $[a,b]$, where
    $$[x,z](t):=(t,x(t),{^C_aD^\alpha_t}x(t),z(t)) \quad \mbox{and} \quad \lambda(t):=\exp\left(-\int_a^t \frac{\partial L}{\partial z}[x,z](\tau)d\tau \right).$$
\end{enumerate}
For any arbitrary but fixed function $x(t)$ and a fixed initial value $z(a)=z_a$
the solution of the differential equation \eqref{MainProblem}: $z[x;t]=z(t,x(t),{^C_aD^\alpha_t}x(t))$ exists, and depends on $t$, $x(t)$ and ${^C_aD^\alpha_t}x(t)$ (see, e.g., \cite{Anosov}). Moreover, under our assumptions, is a $C^2$ function of its arguments.
In order for the equation \eqref{MainProblem} to define a functional,
$z$, of $x(t)$ we must solve equation \eqref{MainProblem} with the same fixed initial condition
$z(a)=z_a$ for all argument functions $x(t)$, and evaluate the solution
$z=z[x;t]$ at the same fixed final time $t=b$ for all argument
functions $x(t)$.

The fractional variational principle of Herglotz (FVPH) is as follows:\\
{\it Let the functional $z=z[x;b]$ of $x(t)$ be given by a differential equation of the form \eqref{MainProblem} and let the function $\eta\in C^1([a,b],\mathbb R^n)$ satisfies the boundary conditions $\eta(a)=\eta(b)=0$ and such that ${^C_aD^\alpha_t}\eta\in C^1([a,b],\mathbb R^n)$. Then the value of the functional $z[x;b]$ is an extremum for function $x(t)$ which satisfy the condition
\begin{equation}\label{principle}
\frac{d}{d\epsilon} \left.z[x+\epsilon\eta;b]\right|_{\epsilon=0}=\frac{d}{d\epsilon} \left.z\left(b,x(b)+\epsilon\eta(b),{^C_aD^\alpha_b}x(b)+\epsilon{^C_aD^\alpha_b}\eta(b)\right)\right|_{\epsilon=0}=0.
\end{equation}}

Observe that, if we introduce a parameter $\epsilon$ in the equation \eqref{MainProblem}, then the solution $z$ still exists and depends also on $\epsilon$, and it is differentiable with respect to $\epsilon$ (see, e.g., Section 2.6 in \cite{Anosov}). Moreover, under our assumptions, $z$ is a $C^2$ function of its arguments.

\begin{Remark} Later we will consider the case where $x(b)$ is free and deduce the respective natural boundary conditions. In this case, the functions $\eta$ that we consider in the formulation of the (FVPH) are such that $\eta(a)=0$ but $\eta(b)$ may take any value.
\end{Remark}

\begin{Remark}
In the case when $\alpha\to1$, FVPH gives the classical variational principle of Herglotz (see \cite{Herglotz}).
\end{Remark}

\begin{Remark}
In the fractional calculus of variations, $L$ does not depend on $z$ and we can integrated from $a$ to $b$, to obtain the functional
$$z[x]=\int_a^b \left[L(t,x(t),{^C_aD^\alpha_t}x(t))+\frac{z_a}{b-a}\right]dt.$$
\end{Remark}

\begin{thm}\label{TNC} Let $x=(x_1,\ldots,x_n)$ be such that $z[x;b]$ defined by equation \eqref{MainProblem} attains an extremum. Then $x$ is a solution of
\begin{equation}\label{NC}\lambda(t)\frac{\partial L}{\partial x_j}[x,z](t)+{_tD^{\alpha_j}_b\left(\lambda(t)\frac{\partial L}{\partial {^C_aD^{\alpha_j}_t}x_j}[x,z](t)\right)}=0,\quad j=1,\ldots n\end{equation}
on $[a,b]$.
\end{thm}
\begin{proof} Let $x$ be such that $z[x;b]$ defined by equation \eqref{MainProblem} attains an extremum. We will show that \eqref{NC} is a consequence of condition \eqref{principle}. The rate of change of $z$, in the direction of $\eta$, is given by
\begin{equation}\label{variation:1}
\theta (t)=\frac{d}{d\epsilon} \left.z[x+\epsilon\eta;t]\right|_{\epsilon=0}=\frac{d}{d\epsilon} \left.z\left(t,x(t)+\epsilon\eta(t),{^C_aD^\alpha_t}x(t)+\epsilon{^C_aD^\alpha_t}\eta(t)\right)\right|_{\epsilon=0}.
\end{equation}
Applying the variation $\epsilon\eta$ to the argument function in equation \eqref{MainProblem} we get
\begin{equation*}\label{var:1}
\frac{d}{dt}z[x+\epsilon\eta;t]=L(t,x(t)+\epsilon\eta(t),{^C_aD^\alpha_t}x(t)+\epsilon{^C_aD^\alpha_t}\eta(t),z[x+\epsilon\eta;t]).
\end{equation*}
Observe that,
\begin{multline*}
\dot{\theta}(t)=\frac{d}{d t}\theta (t)=\frac{d}{dt}\frac{d}{d\epsilon} \left.z\left(t,x(t)+\epsilon\eta(t),{^C_aD^\alpha_t}x(t)+\epsilon{^C_aD^\alpha_t}\eta(t)\right)\right|_{\epsilon=0}\\
=\frac{d}{d\epsilon }\frac{d}{dt} \left.z\left(t,x(t)+\epsilon\eta(t),{^C_aD^\alpha_t}x(t)+\epsilon{^C_aD^\alpha_t}\eta(t)\right)\right|_{\epsilon=0}\\
=\left.\frac{d}{d \epsilon }L\left(t,x(t)+\epsilon\eta(t),{^C_aD^\alpha_t}x(t)+\epsilon{^C_aD^\alpha_t}\eta(t),z[x+\epsilon\eta;t]\right)\right|_{\epsilon=0}.
\end{multline*}
This gives a differential equation for $\theta (t)$:
$$\dot{\theta}(t)-\frac{\partial L}{\partial z}[x,z](t)\theta(t)=\sum_{j=1}^n\left(\frac{\partial L}{\partial x_j}[x,z](t)\eta_j(t)+\frac{\partial L}{\partial {^C_aD^{\alpha_j}_t}x_j}[x,z](t){^C_aD^{\alpha_j}_t}\eta_j(t)\right)$$
whose solution is
$$\theta(t)\lambda(t)-\theta(a)=\int_a^t \sum_{j=1}^n\left(\frac{\partial L}{\partial x_j}[x,z](\tau)\eta_j(\tau)+\frac{\partial L}{\partial {^C_aD^{\alpha_j}_\tau}x_j}[x,z](\tau){^C_aD^{\alpha_j}_\tau}\eta_j(t)\right)\lambda(\tau)d\tau$$
with notation $\lambda(t)=\exp\left(-\int_a^t \frac{\partial L}{\partial z}[x,z](\tau)d\tau\right)$.
 Observe that $\theta(a)=0$. Indeed, as explained earlier, we evaluate the solution of the equation \eqref{MainProblem} with the same fixed initial condition $z(a)$, independently of the function $x(t)$.
At $t=b$, we get
$$\theta(b)\lambda(b)=\int_a^b \sum_{j=1}^n\left(\frac{\partial L}{\partial x_j}[x,z](t)\eta_j(t)+\frac{\partial L}{\partial {^C_aD^{\alpha_j}_t}x_j}[x,z](t){^C_aD^{\alpha_j}_t}\eta_j(t)\right)\lambda (t)dt$$
Observe that $\theta(b)$ is the variation of $z[x;b]$. If  $x$ is such that $z[x;b]$ defined by equation \eqref{MainProblem} attains an extremum, then $\theta(b)$ is identically zero. Hence, we get
$$\int_a^b\lambda(t)\sum_{j=1}^n\left(\frac{\partial L}{\partial x_j}[x,z](t)\eta_j(t)+\frac{\partial L}{\partial {^C_aD^{\alpha_j}_t}x_j}[x,z](t){^C_aD^{\alpha_j}_t}\eta_j(t)\right)dt=0.$$
Integrating by parts (cf. Theorem \ref{IP}) we get
\begin{multline*}
\int_a^b\sum_{j=1}^n\left[\lambda(t)\frac{\partial L}{\partial x}_j[x,z](t)+{_tD^{\alpha_j}_b}\left(\lambda(t)\frac{\partial L}{\partial {^C_aD^{\alpha_j}_t}x_j}[x,z](t)\right)\right]\eta_j(t)dt\\
+\sum_{j=1}^n\left[\eta_j(t){_tI^{1-\alpha_{j}}_b}\left(\lambda(t)\frac{\partial L}{\partial {^C_aD^{\alpha_j}_t}x_j}[x,z](t)\right)\right]_a^b=0.
\end{multline*}
Since $\eta(a)=\eta(b)=0$, and $\eta$ is an arbitrary function elsewhere, equation \eqref{NC} follows by the fundamental lemma of the calculus of variations.
\end{proof}

\begin{Remark} The function $\theta$ in Eq. \eqref{variation:1} is well defined (see, e.g., Section 2.6 in \cite{Anosov}).
\end{Remark}

\begin{Remark} Let $\alpha_j$ goes to $1$, for all $j=1,\ldots,n$ and the Lagrangian $L$ is of the $C^2$ class. Then equations \eqref{NC} become
$$\lambda(t)\frac{\partial L}{\partial x_j}[x,z](t)-\frac{d}{dt}\left(\lambda(t)\frac{\partial L}{\partial \dot{x}_j}[x,z](t)\right)=0, \quad  j=1,\ldots,n.$$
Then, since $\dot{\lambda}(t)=\displaystyle-\frac{\partial L}{\partial z}[x,z](t)\lambda(t)$, we get
$$\lambda(t)\left[\frac{\partial L}{\partial x_j}[x,z](t)+\frac{\partial L}{\partial z}[x,z](t)\frac{\partial L}{\partial \dot{x}_j}[x,z](t)-\frac{d}{dt}\frac{\partial L}{\partial \dot{x}_j}[x,z](t)\right]=0$$
if and only if
\begin{equation}\label{CH}
\frac{\partial L}{\partial x_j}[x,z](t)+\frac{\partial L}{\partial z}[x,z](t)\frac{\partial L}{\partial \dot{x}_j}[x,z](t)-\frac{d}{dt}\frac{\partial L}{\partial \dot{x}_j}[x,z](t)=0, \quad j=1,\ldots,n.
\end{equation}
To equations \eqref{CH}, Herglotz called the generalized Euler--Lagrange equations \cite{Herglotz}.
\end{Remark}

\begin{ex} Consider the system

$$\left\{ \begin{array}{l}
\dot{z}(t)=\left[{^C_0D^{0.5}_t}x(t)-\frac{2}{\Gamma(2.5)}t^{1.5}\right]^2, \quad \mbox{ with } \, t\in[0,1],\\
z(0)=0,\\
x(0)=0,\, x(1)=1.
\end{array}\right.$$

We wish to find a curve $x$ for which $z[x;1]$ attains the minimum value. Observe that for all $x$, we have that $z$ is an increasing function and thus $z(t)\geq0$, for all $t\in[0,1]$.

 Let $\overline x(t)=t^2$. Then ${^C_0D^{0.5}_t}\overline x(t)=\frac{2}{\Gamma(2.5)}t^{1.5}$ (see e.g. \cite{kilbas}). For such choice of $x$ we get the differential equation $\dot{z}(t)=0$. As $z(0)=0$ the unique solution is $\overline z(t)=0$. Therefore $z[x;1]$ obtain the minimum value, which is $0$, at $\overline x$. In this case $\lambda(t)=1$ and equation \eqref{NC} takes the form
$${_tD_1^{0.5}}\left[{^C_0D^{0.5}_t}x(t)-\frac{2}{\Gamma(2.5)}t^{1.5}\right]=0$$
which is satisfied for $\overline x$.
\end{ex}

\begin{ex} Consider the system with the Lagrangian depending on $z$:

$$\left\{ \begin{array}{l}
\dot{z}(t)=\left[{^C_0D^{0.5}_t}x(t)-\frac{2}{\Gamma(2.5)}t^{1.5}\right]^2\exp(t)+z, \quad \mbox{ with } \, t\in[0,1],\\
z(0)=1,\\
x(0)=0, \, x(1)=1.\\
\end{array}\right.$$

In this case we have $\lambda(t)=\exp(-t)$ and equation \eqref{NC} takes the form
\begin{equation}\label{ex2:1}
{_tD_1^{0.5}}\left[\exp(-t)\left({^C_0D^{0.5}_t}x(t)-\frac{2}{\Gamma(2.5)}t^{1.5}\right)\exp(t)\right]=0.
\end{equation}
Observe that the function $\overline x(t)=t^2$ is a solution to equation \eqref{ex2:1}, but we cannot conclude that it is a minimizer (or maximizer) of the functional $z[x;1]$.
\end{ex}

Note that in order to uniquely determine the unknown function $x$ which is a solution to equation \eqref{NC}, we must consider the system of differential equations
$$\left\{ \begin{array}{l}
\displaystyle \frac{dz}{dt}=L[x,z](t),\\
\displaystyle\lambda(t)\frac{\partial L}{\partial x_j}[x,z](t)+{_tD^{\alpha_j}_b\left(\lambda(t)\frac{\partial L}{\partial {^C_aD^{\alpha_j}_t}x_j}[x,z](t)\right)}=0,\,
\end{array}\right.$$
 with $\lambda(t)=\exp\left(-\int_a^t \frac{\partial L}{\partial z}[x,z](\tau)d\tau \right)$ together with the boundary condition: $z(a)=z_a$, $x(a)=x_a$, $x(b)=x_b$, where $z_a\in \mathbb R$ and $x_a,x_b\in \mathbb R^n$.
In the case where $x_j(b)$, $j\in\{1,\ldots,n\}$, is not specified we need an additional condition known as the transversality condition.

\begin{thm}\label{TNC2}  Let $x$ be such that $z[x;b]$ defined by equation \eqref{MainProblem} attains an extremum.  Then, $x$ is a solution to the system of fractional differential equations \eqref{NC}. If $x_j(b)$, $j\in\{1,\ldots,n\}$, is not specified, then the transversality condition
\begin{equation}\label{trans}
{_tI^{1-\alpha_j}_b}\left(\lambda(t)\frac{\partial L}{\partial {^C_aD^{\alpha_j}_t}x_j}[x,z](t)\right)=0 \quad \mbox{at} \quad t=b
\end{equation}
holds.
\end{thm}
\begin{proof} The only difference with respect to the proof of Theorem~\ref{TNC} is that $\eta_j(b)$ may vanish or not. Assuming first that $\eta_j(b)=0$ we deduce equations \eqref{NC}. Therefore $$\sum_{j=1}^n\left[\eta_j(t){_tI^{1-\alpha_{j}}_b}\left(\lambda(t)\frac{\partial L}{\partial {^C_aD^{\alpha_j}_t}x_j}[x,z](t)\right)\right]_a^b=0.$$
Since $\eta(a)=0$,  $\eta_i(b)=0$ for $i\neq j$ and $\eta_j(b)$ is arbitrary equation \eqref{trans} follows.
\end{proof}

When dealing with fractional operators usually we do not need to assume that functions have a nice behavior in the sense of smoothness properties. From the other hand, the drawback is that equations of type \eqref{NC} are in most cases impossible to be solved analytically, and to overcome this problem numerical methods are applied (see, e.g., \cite{Atan2,ford1,ford2,Pooseh}). One of these methods consists in replacing the fractional operator by an expansion that involves only integer order derivatives \cite{Samko}.
Let $\alpha\in(0,1)$, $(c,d)$ be an open interval in $\mathbb R$, and $[a,b]\subset(c,d)$ be such that for each $t\in[a,b]$ the closed ball
$B_{b-a}(t)$, with center at $t$ and radius $b-a$, lies in $(c,d)$. If $x$ is an analytic function in $(c,d)$ then
\begin{equation}\label{approx1}{_aD_t^\alpha}x(t)=\sum_{k=0}^\infty \binom{\alpha}{k}\frac{(t-a)^{k-\alpha}}{\Gamma(k+1-\alpha)}x^{(k)}(t),
\quad \text{ where } \binom{\alpha}{k}=\frac{(-1)^{k-1}\alpha\Gamma(k-\alpha)}{\Gamma(1-\alpha)\Gamma(k+1)}.\end{equation}
Using the well-known relations between the two type of fractional  derivatives,
\begin{equation}\label{approx2}{_a^CD_t^\alpha}x(t)={_aD_t^\alpha}x(t)-\frac{x(a)}{\Gamma(1-\alpha)}(t-a)^{-\alpha},\end{equation}
similar formula is given for the Caputo fractional derivative. \\
For simplicity, let us assume that $n=1$ and fix $N\in\mathbb N$. If $x\in C^{2N}([a,b],\mathbb R)$, $L\in C^{N+1}([a,b]\times \mathbb R^{3} ,\mathbb R)$, then by replacing the fractional derivative in equation \eqref{MainProblem} by formulas \eqref{approx1}--\eqref{approx2}, with the approximation up to order $N$, we obtain
$$\dot{z}(t)=L\left(t,x(t), \sum_{k=0}^N \binom{\alpha}{k}\frac{(t-a)^{k-\alpha}}{\Gamma(k+1-\alpha)}x^{(k)}(t)-\frac{x_a}{\Gamma(1-\alpha)}(t-a)^{-\alpha},z(t)\right)$$
with $t\in[a,b]$.
Define
$$ \overline L(t,x(t),\dot{x}(t),\ldots,x^{(N)}(t),z(t))$$
$$:=L\left(t,x(t), \sum_{k=0}^N \binom{\alpha}{k}\frac{(t-a)^{k-\alpha}}{\Gamma(k+1-\alpha)}x^{(k)}(t)-\frac{x_a}{\Gamma(1-\alpha)}(t-a)^{-\alpha},z(t)\right).$$
Therefore, we got the higher-order variational problems of Herglotz (see \cite{santos})
$$\dot{z}(t)=\overline L(t,x(t),\dot{x}(t),\ldots,x^{(N)}(t),z(t)), \quad \mbox{ with } \, t\in[a,b].$$
 As it was proved in \cite{santos},
if $x_N$ is such that $z_N[x_N;b]$ attains an extremum, then $(x_N,z_N)$ satisfies the Euler--Lagrange equation
$$\sum_{k=0}^N (-1)^k\frac{d^k}{dt^k}\left(\lambda_N(t)\frac{\partial\overline L}{\partial x^{(k)}}(t,x(t),\dot{x}(t),\ldots,x^{(N)}(t),z(t))\right)=0,$$
on $[a,b]$, where $\displaystyle \lambda_N(t)=\exp\left(-\int_a^t \frac{\partial \overline L}{\partial z}(\tau,x_N(\tau),\dot{x}_N(\tau),\ldots,x_N^{(N)}(\tau),z_N(\tau))d\tau \right)$.

\begin{ex} \label{aprox} Consider the system
$$\left\{ \begin{array}{l}
\dot{z}(t)=\left[{^C_0D^{0.5}_t}x(t)-\frac{2}{\Gamma(2.5)}t^{1.5}\right]^2, \quad \mbox{ with } \, t\in[0,1],\\
z(0)=0,\\
x(0)=0, \, x(1)=1.
\end{array}\right.$$
We replace the Caputo fractional derivative by the truncated sum
$${_a^CD_t^\alpha}x(t)\approx \sum_{k=0}^N \binom{\alpha}{k}\frac{(t-a)^{k-\alpha}}{\Gamma(k+1-\alpha)}x^{(k)}(t)-\frac{x(a)}{\Gamma(1-\alpha)}(t-a)^{-\alpha},
$$
where $N\in\mathbb N$ depends on the number of given boundary conditions. Since in the considered system $x(0)=0$ and $x(1)=1$, we take $N=1$.
Therefore, we get
$$\dot{z}(t)=\left[\sum_{k=0}^1 \binom{\alpha}{k}\frac{(t-a)^{k-\alpha}}{\Gamma(k+1-\alpha)}x^{(k)}(t))-\frac{2}{\Gamma(2.5)}t^{1.5}\right]^2.$$
Applying the classical Euler--Lagrange equation we obtain a second order differential equation, whose solution in shown on Figure~\ref{ex}.

\begin{figure}\centering
  \includegraphics[width=8cm]{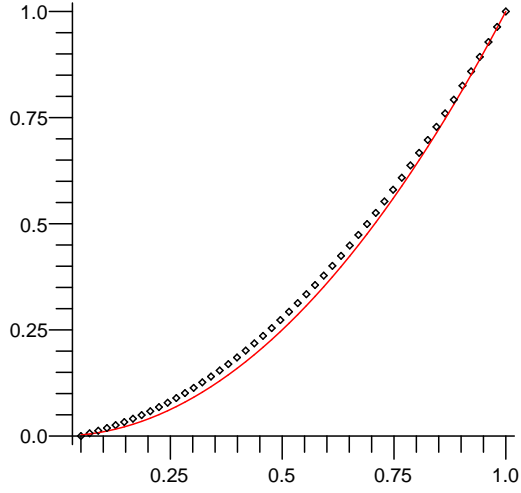}
\caption{Numerical solution (dot line) vs exact solution (continuous line) in Example~\ref{aprox}.}\label{ex}
\end{figure}

\end{ex}

\subsection{Higher-order fractional derivatives}

In this subsection, we generalize the fractional Herglotz problem by considering higher-order fractional derivatives.

\begin{thm}\label{TNC5} Consider the system
\begin{equation}\label{System3}\left\{ \begin{array}{l}
\dot{z}(t)=L(t, x(t),{^C_aD^{\alpha}_t}x(t),z(t)), \quad \mbox{ with } \, t\in[a,b],\\
z(a)=z_a,\\
x_i^{(j)}(a)=x_{ia}^{j}, \, x_i^{(j)}(b)=x_{ib}^{j},
\end{array}\right.\end{equation}
with for all $i\in\{1,\ldots,n\}$ and $j\in\{0,\ldots,i-1\}$, where we assume that the following conditions:
\begin{enumerate}
\item $\alpha_i\in(i-1,i)$,
\item $x_i\in C^n([a,b],\mathbb R)$, ${^C_aD^\alpha_t}x_i\in C^1([a,b],\mathbb R)$,
\item $z_a,x_{ia}^{j},x_{ib}^{j}$ are fixed reals,
\item the Lagrangian $L:[a,b]\times\mathbb R^{2n+1}\to\mathbb R$ is of class $C^1$ and the maps
 $t\mapsto \displaystyle\lambda(t)\frac{\partial L}{\partial {^C_aD^{\alpha_i}_t}x_i}[x,z](t)$, $i\in\{1,\ldots,n\}$, are such that $\displaystyle{_tD^{\alpha_i}_b\left(\lambda(t)\frac{\partial L}{\partial {^C_aD^{\alpha_i}_t}x_i}[x,z](t)\right)}$ exist and are continuous on $[a,b]$, where
    $$[x,z](t):=(t,x(t),{^C_aD^{\alpha}_t}x(t),z(t))$$
     $$ \mbox{and} \quad \lambda(t)=\exp\left(-\int_a^t \frac{\partial L}{\partial z}[x,z](\tau)d\tau \right).$$
\end{enumerate}
hold. If  $x=(x_1,\ldots,x_n)$ is such that $z[x;b]$ defined by \eqref{System3} attains an extremum, then $(x_1,\ldots,x_n,z)$ satisfies the system of equations
\begin{equation}\label{NCHO}
\lambda(t)\frac{\partial L}{\partial x_i}[x,z](t)+{_tD^{\alpha_i}_b\left(\lambda(t)\frac{\partial L}{\partial {^C_aD^{\alpha_i}_t}x_i}[ x,z](t)\right)}=0, \quad i=1,\ldots,n,
\end{equation}
on $[a,b]$.
\end{thm}

\begin{proof} Let $x$ be such that $z[x;b]$ defined by \eqref{System3} attains an extremum. Consider variations of type $x+\epsilon\eta$ with $\eta=(\eta_1,\ldots,\eta_n)$ satisfying $\eta_i^{(j)}(a)=\eta_i^{(j)}(b)=0$, for all $i\in\{1,\ldots,n\}$ and $j\in\{0,\ldots,i-1\}$. The rate of change of $z$ in the direction of $\eta$ is given by
$$\theta (t)=\frac{d}{d\epsilon} \left.z[x+\epsilon\eta;t]\right|_{\epsilon=0}\frac{d}{d\epsilon} = \left.z\left(t,x(t)+\epsilon\eta(t),{^C_aD^\alpha_t}x(t)+\epsilon{^C_aD^\alpha_t}\eta(t)\right)\right|_{\epsilon=0}.$$
Applying $\epsilon\eta$ to differential equation in \eqref{System3} we get
$$\frac{d}{dt}z[x+\epsilon\eta;t]=L(t, x(t)+\epsilon\eta(t),{^C_aD^{\alpha}_t}x(t)+\epsilon{^C_aD^{\alpha}_t}\eta(t),z[x+\epsilon\eta;t]).$$
Then, as in the proof of Theorem~\ref{NC}, we obtain the linear ODE with respect to $\theta$:
$$\dot{\theta}(t)=\sum_{i=1}^n\left[\frac{\partial L}{\partial x_i}[x,z](t)\eta_i(t)+\frac{\partial L}{\partial {^C_aD^{\alpha_i}_t}x_i}[ x](t){^C_aD^{\alpha_i}_t}\eta_i(t)\right]+\frac{\partial L}{\partial z}[x,z](t)\theta(t).$$
Solving the linear ODE and using the fact that $\theta(a)=\theta(b)=0$ we get
$$\int_a^b\lambda(t)\sum_{i=1}^n\left[\frac{\partial L}{\partial x_i}[x,z](t)\eta_i(t)+\frac{\partial L}{\partial {^C_aD^{\alpha_i}_t}x_i}[x,z](t){^C_aD^{\alpha_i}_t}\eta_i(t)\right]dt=0.$$
Integrating by parts we obtain
$$\int_a^b\sum_{i=1}^n\left[\lambda(t)\frac{\partial L}{\partial x_i}[x,z](t)+{_tD^{\alpha_i}_b}\left(\lambda(t)\frac{\partial L}{\partial {^C_aD^{\alpha_i}_t}x_i}[x,z](t)\right)\right]\eta_i(t)dt$$
$$+\left[\sum_{i=1}^n\sum_{j=0}^{i-1}(-1)^{i-1-j}\eta_i^{(i-1-j)}(t){_tD^{\alpha_i+j-i}_b}\left(\lambda(t)\frac{\partial L}{\partial {^C_aD^{\alpha_i}_t}x_i}[x,z](t)\right)\right]_a^b=0.$$
Since $\eta_i^{(j)}(a)=\eta_i^{(j)}(b)=0$ for all $i\in\{1,\ldots,n\}$ and $j\in\{0,\ldots,i-1\}$, and $\eta$ is arbitrary  elsewhere, the theorem is proven.
\end{proof}

Next theorem generalizes transversality conditions given in Theorem~\ref{TNC2}.

\begin{thm}\label{TNC6} Consider system \eqref{System3} with an exception that $x_i^{(j)}(b)$ may take any value, for all $i\in\{1,\ldots,n\}$ and $j\in\{0,\ldots,i-1\}$. If $x=(x_1,\ldots,x_n)$ is such that $z[x;b]$ attains an extremum, then $(x_1,\ldots,x_n,z)$ satisfies the system of equations \eqref{NCHO} on $[a,b]$, and the transversality conditions
$${_tD^{\alpha_i+j-i}_b}\left(\lambda(t)\frac{\partial L}{\partial {^C_aD^{\alpha_i}_t}x_i}[x,z](t)\right)=0 \quad \mbox{at} \quad t=b,$$
for all $i\in\{1,\ldots,n\}$ and all $j\in\{0,\ldots, i-1\}$.
\end{thm}


\subsection{Several independent variables}

The fractional variational principle of Herglotz can be generalized to one involving several independent variables. In the following, the independent variables will be the time variable $t\in[a,b]$ and the spacial coordinates $x=(x_1,\ldots,x_n)\in \Omega=\prod_{i=1}^n[a_i,b_i]$, $n\in\mathbb N$. The argument function of the functional $z[u;b]$ defined by the variational principle will be $u=(u_1(t,x),\ldots,u_m(t,x))$, $m\in\mathbb N$. In what follows, for $\alpha=(\alpha_1,\ldots,\alpha_m)$ and $\beta_i=(\beta_{i,1},\ldots,\beta_{i,m})$, $i=1,\ldots,n$ we will use the following notation:
$${^C_aD^{\alpha}_t}u=({^C_aD^{\alpha_1}_t}u_1,\ldots,{^C_aD^{\alpha_m}_t}u_m),$$
$${^C_{a_i}D^{\beta_i}_{x_i}}u=({^C_{a_i}D^{\beta_{i,1}}_{x_i}}u_1,\ldots,{^C_{a_i}D^{\beta_{i,m}}_{x_i}}u_m),\quad i=1,\ldots,n.$$
The fractional variational principle with several independent variables is as follows:\\
{\it Let the functional $z=z[u;b]$ of $u=u(t,x)$ be given by an integro-differential equation of the form
\begin{equation}\label{GVPS}
\frac{d z}{dt}=\displaystyle \int_\Omega L(t,x,u(t,x),{^C_aD^{\alpha}_t}u(t,x),{^C_{a_1}D^{\beta_1}_{x_1}}u(t,x),\ldots,{^C_{a_n}D^{\beta_n}_{x_n}}u(t,x),z(t))\, d^nx,
\end{equation}
$t\in[a,b]$, where $d^nx=dx_1\ldots dx_n$; and let the following conditions be satisfied
\begin{enumerate}
\item $u(t,x)=g(t,x)$ for $(t,x)\in \partial P$, where $P=[a,b]\times\Omega$, $\partial P$ is the boundary of $P$ and $g:\partial P\rightarrow R^m$ is a given function;
\item  $\alpha_j, \beta_{i,j}\in(0,1)$,
\item  $u_j\in C^1(P,\mathbb R)$, ${^C_aD^{\alpha_j}_t}u_j, {^C_{a_i}D^{\beta_{i,j}}_{x_i}}u_j \in C^1(P,\mathbb R)$,
\item $L:P\times\mathbb R^{m(n+2)+1}\to\mathbb R$ is of class $C^1$ and the maps\\
$(t,x)\mapsto \displaystyle\lambda(t)\frac{\partial L}{\partial {^C_aD^{\alpha_j}_t}u_j}[u,z](t,x)$ and $(t,x)\mapsto \displaystyle \lambda(t)\frac{\partial L}{\partial {^C_{a_i}D^{\beta_{i,j}}_{x_i}}u_j}[u,z](t,x)$ \\
are such that
$\displaystyle {_tD^{\alpha_j}_b}\left(\lambda(t)\frac{\partial L}{\partial {^C_aD^{\alpha_j}_t}u_j}[u,z](t,x)\right)$ and
$\displaystyle _{x_i}D^{\beta_{i,j}}_{b_i} \left(\lambda(t)\frac{\partial L}{\partial {^C_{a_i}D^{\beta_{i,j}}_{x_i}}u_j}[u,z](t)\right)$
 exist and are continuous on $P$,
where
$$[u,z](t,x):=(t,x,u(t,x),{^C_aD^{\alpha}_t}u(t,x),{^C_{a_1}D^{\beta_1}_{x_1}}u(t,x),\ldots,{^C_{a_n}D^{\beta_n}_{x_n}}u(t,x),z(t),$$
$$\lambda(t):=\exp\left(-\int_a^t\int_\Omega  \frac{\partial L}{\partial z}[u,z](\tau,x) d^nx d\tau \right).$$
\end{enumerate}
for all $i\in\{1,\ldots,n\}$ and $j\in\{0,\ldots,m\}$.

Consider $\eta\equiv (\eta_1(t,x),\ldots,\eta_m(t,x))\in C^1([a,b]\times\Omega,\mathbb R^m)$ that satisfies the boundary conditions
$\eta(t,x)=0$ for $(t,x)\in \partial P$ and such that ${^C_aD^{\alpha_j}_t}\eta_j, {^C_{a_i}D^{\beta_{i,j}}_{x_i}}\eta_j \in C^1(P,\mathbb R)$. Then, the value of the functional $z[u;b]$ is an extremum for function $u$  which satisfy the condition
\begin{equation}\label{principle:md}
\frac{d}{d\epsilon} \left.z[u+\epsilon\eta;b]\right|_{\epsilon=0}=\frac{d}{d\epsilon} \left.z(b,[u](b,x))\right|_{\epsilon=0}=0,
\end{equation}}
where $[u](t,x):=(t,x,u(t,x),{^C_aD^{\alpha}_t}u(t,x),{^C_{a_1}D^{\beta_1}_{x_1}}u(t,x),\ldots,{^C_{a_n}D^{\beta_n}_{x_n}}u(t,x)).$
It should be pointed that when a variation $\epsilon\eta$ is applied to $u$ the equation \eqref{GVPS}, defining the functional $z$,
must be solved with the same fixed initial condition $z_a$ at $t=a$ and the solution evaluated at the
same fixed final time $t=b$ for all varied argument functions $\epsilon\eta$.

\begin{thm}\label{TNC7}
If $u$ is such that the functional $z[u;b]$ defined by equation \eqref{GVPS} attains an extremum,
then  $u$ is a solution of the system of equations
\begin{multline}\label{EL:md}
\lambda(t)\frac{\partial L}{\partial u_j}[u,z](t,x)
+{_tD^{\alpha_j}_b}\left(\lambda(t)\frac{\partial L}{\partial {^C_aD^{\alpha_j}_t}u_j}[u,z](t,x)\right)\\
+\sum_{i=1}^n {_{x_i}D^{\beta_{i,j}}_{b_i}} \left(\lambda(t)\frac{\partial L}{\partial {^C_{a_i}D^{\beta_{i,j}}_{x_i}}u_j}[u,z](t,x)\right)=0,
\end{multline}
$j=1,\ldots,m$.
\end{thm}

\begin{proof}
The idea of the proof is similar to that of the proof of Theorem \ref{TNC}. We show that equations \eqref{EL:md} are a consequence of condition \eqref{principle:md}. The rate of change of $z$, in the direction of $\eta$, is given by
\begin{equation}\label{variation}
\theta (t)=\frac{d}{d\epsilon} \left.z[u+\epsilon\eta;t]\right|_{\epsilon=0}=\frac{d}{d\epsilon} \left.z(t,[u](t,x))\right|_{\epsilon=0}.
\end{equation}
Applying the variation $\epsilon\eta$ to the argument function in equation \eqref{GVPS}, differentiating with respect to $\epsilon$ and then setting $\epsilon=0$ we get a differential equation for $\theta (t)$:
$$\frac{d\theta}{dt}(t)= \int_\Omega\left[\sum_{j=1}^m \frac{\partial L}{\partial u_j}[u,z](t,x)\eta_j
+\sum_{j=1}^m\frac{\partial L}{\partial {^C_aD^{\alpha_j}_t}u_j}[u,z](t,x){^C_aD^{\alpha_j}_t}\eta_j\right.$$
$$\left.+\sum_{i=1}^n\sum_{j=1}^m\frac{\partial L}{\partial {^C_{a_i}D^{\beta_{i,j}}_{x_i}}u_j}[u,z](t,x){^C_{a_i}D^{\beta_{i,j}}_{x_i}}\eta_j
+\frac{\partial L}{\partial z}[u,z](t,x)\theta(t)\right]d^nx.$$
Denoting
\begin{multline*}
A(t)=\int_\Omega\sum_{j=1}^m \left[\frac{\partial L}{\partial u_j}[u,z](t,x)\eta_j
+\frac{\partial L}{\partial {^C_aD^{\alpha_j}_t}u_j}[u,z](t,x){^C_aD^{\alpha_j}_t}\eta_j\right.\\
\left.+\sum_{i=1}^n\frac{\partial L}{\partial {^C_{a_i}D^{\beta_{i,j}}_{x_i}}u_j}[u,z](t,x){^C_{a_i}D^{\beta_{i,j}}_{x_i}}\eta_j\right]d^nx
\end{multline*}
and
$$B(t)=\int_\Omega  \frac{\partial L}{\partial z}[u,z](t,x) d^nx.$$
we get
\begin{equation}\label{ODE}
\frac{d\theta}{dt}(t)-B(t)\theta(t)=A(t).
\end{equation}
Solving equation \eqref{ODE} and taking into consideration that $\theta(a)=\theta(b)=0$ (see the proof of Theorem \ref{TNC}) we obtain
$$\int_a^b\lambda(t)A(t)dt=0.$$
Integrating by parts we get
$$\int_a^b\int_\Omega\sum_{j=1}^m \left[\lambda(t)\frac{\partial L}{\partial u_j}[u,z](t,x)
+{_tD^{\alpha_j}_b}\left(\lambda(t)\frac{\partial L}{\partial {^C_aD^{\alpha_j}_t}u_j}[u,z](t,x)\right)\right.$$
$$\left.+\sum_{i=1}^n {_{x_i}D^{\beta_{i,j}}_{b_i}} \left(\lambda(t)\frac{\partial L}{\partial {^C_{a_i}D^{\beta_{i,j}}_{x_i}}u_j}[u,z](t,x)\right)\right]\eta_j d^nxdt=0$$
as $\eta$ satisfies the boundary conditions $\eta(a,x)=\eta(b,x)=0$ for $x\in \Omega$, and $\eta(t,x)=0$  for $x\in\partial \Omega$, $a\leq t \leq b$. Finally, since $\eta_j$ are arbitrary functions, it follows that for all $j\in\{1,\ldots,m\}$,
\begin{multline}
\lambda(t)\frac{\partial L}{\partial u_j}[u,z](t,x)
+{_tD^{\alpha_j}_b}\left(\lambda(t)\frac{\partial L}{\partial {^C_aD^{\alpha_j}_t}u_j}[u,z](t,x)\right)+\\
\sum_{i=1}^n {_{x_i}D^{\beta_{i,j}}_{b_i}} \left(\lambda(t)\frac{\partial L}{\partial {^C_{a_i}D^{\beta_{i,j}}_{x_i}}u_j}[u,z](t,x)\right)=0
\end{multline}
on $[a,b]\times\Omega$.
\end{proof}

\begin{Remark}
We note that it is straightforward to obtain the transversality conditions for the Herglotz's fractional variational problems with several variables (cf. \cite{book:MT}).
\end{Remark}

\section{Noether-type theorem}

Symmetric principles are key issues in mathematics and physics. There are close relationships
between symmetries and conserved quantities. Noether first proposed the famous Noether symmetry theorem which describes the
universal fact that invariance of the action functional with respect to some family of parameter transformations
gives rise to the existence of certain conservation laws, i.e., expressions
preserved along the solutions of the Euler-Lagrange equation. In this Section we extend Noether's theorem to one that holds for fractional variational problems of Herglotz-type with one (Theorem~\ref{Noether:without:time}) or several independent variables (Theorem~\ref{Noether:without:time:m}).

Consider the one-parameter family of transformations
\begin{equation}\label{tra:fam}
\bar{x}_j=h_j(t,x,s), \quad j=1,\ldots,n,
\end{equation}
depending on a parameter $s$, $s\in (-\varepsilon,\varepsilon)$, where  $h_j$ are of class $C^2$ and such that $h_j(t,x,0)=x_j$ for all $j\in\{1,\ldots,n\}$ and for all $(t,x)\in [a,b]\times \mathbb R^n$. By Taylor's formula we have
\begin{equation}\label{tra:fam:T}
h_j(t,x,s)=h_j(t,x,0)+s\xi_j(t,x)+o(s)=x_j+s\xi_j(t,x)+o(s),
\end{equation}
provided that $|s|\leq \varepsilon$, where $\xi_j(t,x)=\frac{\partial}{\partial s}h_j(t,x,s)|_{s=0}$. For $|s|\leq \varepsilon$ the linear approximation to transformation \eqref{tra:fam} is $\bar{x}_j\approx x_j+s\xi_j(t,x)$. \\
This transformation leads, for any fixed
sufficiently small parameter $s$ and any fixed function $x(\cdot)=(x_1(\cdot),\ldots,x_n(\cdot))$, to
$\overline x(\cdot)=( \overline x_1(\cdot),\ldots,\overline x_n(\cdot))$ given by
$$\bar{x}_j(t)=h_j(t,x(t),s), \quad j=1,\ldots,n.$$

Denote by $\theta=\theta(t)$ the total variation produced by the family of transformations \eqref{tra:fam:T}, that is

\begin{equation}\label{variation:1t}
\theta (t)=\frac{d}{ds} \left.\bar{z}[x+s\xi;t]\right|_{s=0}.
\end{equation}

We define the invariance of the functional $z$, defined by differential equation \eqref{MainProblem}, under transformation \eqref{tra:fam:T} as follows.

\begin{Definition}\label{def:in}
The transformation \eqref{tra:fam:T} leaves the functional $z$, defined by differential equation \eqref{MainProblem}, invariant if $\theta (t)\equiv0$.
\end{Definition}

\begin{Remark}
When $s=0$, the following holds $\bar{z}[\bar{x};t]=z[x;t]$.
\end{Remark}

\begin{thm}\label{Noether:without:time}
If the functional $z$ defined by differential equation \eqref{MainProblem} is invariant in the sense of definition~\ref{def:in}, then
\begin{equation}\label{eq:Noether:without:time}
\sum_{j=1}^n\mathcal{D}^{\alpha_j}\left[\lambda(t)\frac{\partial L}{\partial {^C_aD^{\alpha_j}_t}x_j}[x,z](t),\xi_j(t,x)\right]=0,
\end{equation}
where $\lambda(t):=\exp\left(-\int_a^t \frac{\partial L}{\partial z}[x,z](\tau)d\tau \right)$ and $\mathcal{D}^{\alpha}[f,g]:=f {^C_aD^{\alpha}_t}\cdot g- g\cdot{_tD^{\alpha}_b}f$, holds along the solutions of the generalized fractional Euler--Lagrange equations \eqref{NC}.
\end{thm}

\begin{proof}
Applying transformations \eqref{tra:fam:T}  to equation \eqref{MainProblem} we get
\begin{equation}\label{var:1:n}
\frac{d}{dt}\bar{z}=L(t,\bar{x}(t),{^C_aD^\alpha_t}\bar{x}(t),\bar{z}(t)).
\end{equation}
Differentiating both sides of \eqref{var:1:n} with respect to $s$ and then setting $s=0$, we get a differential equation for $\theta (t)$
$$\dot{\theta}(t)-\frac{\partial L}{\partial z}[x,z](t)\theta(t)=\sum_{j=1}^n\left(\frac{\partial L}{\partial x_j}[x,z](t)\xi_j(t)+\frac{\partial L}{\partial {^C_aD^{\alpha_j}_t}x_j}[x,z](t){^C_aD^{\alpha_j}_t}\xi_j(t)\right)$$
whose solution is
$$\theta(t)\lambda(t)-\theta(a)=\int_a^t \sum_{j=1}^n\left(\frac{\partial L}{\partial x_j}[x,z](\tau)\xi_j(\tau)+\frac{\partial L}{\partial {^C_aD^{\alpha_j}_\tau}x_j}[x,z](\tau){^C_aD^{\alpha_j}_\tau}\xi_j(t)\right)\lambda(\tau)d\tau$$
with notation $\lambda(t)=\exp\left(-\int_a^t \frac{\partial L}{\partial z}[x,z](\tau)d\tau\right)$.
 Observe that $\theta(a)=0$,
also by hypothesis $\theta (\tau)=0$ for arbitrary $\tau$. Thus
$$0=\int_a^ \tau\sum_{j=1}^n\left(\frac{\partial L}{\partial x_j}[x,z](t)\xi_j(t)+\frac{\partial L}{\partial {^C_aD^{\alpha_j}_t}x_j}[x,z](t){^C_aD^{\alpha_j}_t}\xi_j(t)\right)\lambda (t)dt$$
On the solutions of the generalized fractional Euler--Lagrange equations \eqref{NC} we have
\begin{equation*}
0=\sum_{j=0}^n \left[\lambda(t)\frac{\partial L}{\partial ^C_aD^{\alpha_j}_t x_j[x,z](t)}
{^C_aD^{\alpha_j}_t}\xi_j(t,x)-
_tD^{\alpha_j}_b\left(\lambda(t)\frac{\partial L}{\partial ^C_aD^{\alpha_j}_tx_j[x,z](t)}\right)\xi_j(t,x)\right].
\end{equation*}
By definition of operator $\mathcal{D}^{\alpha}$ we obtain equation \eqref{eq:Noether:without:time}.
\end{proof}

\begin{ex}
Consider the Lagrangian in one spatial dimension
\begin{equation}\label{ex:l:L}
L(t,x(t),{^C_aD^\alpha_t}x(t),z(t))=f(^C_aD^{\alpha}_tx)-\gamma z,
\end{equation}
where $f$ is a $C^1$ function and $\gamma$ is a positive constant. Obviously, the Lagrangian \eqref{ex:l:L} is invariant under transformation
$$\bar{x}(t)=x(t)+c,$$
where $c$ is a constant. Therefore, Theorem~\ref{Noether:without:time} indicates that
\begin{equation}\label{ex:l:n}
\mathcal{D}^{\alpha}\left[e^{\gamma t}f'(^C_aD^{\alpha}_tx),c\right]=0,
\end{equation}
along any solution of $_tD^{\alpha}_b\left(e^{\gamma t}f'(^C_aD^{\alpha}_tx)\right)=0$. Notice that equation \eqref{ex:l:n} can be written in the form $\frac{d}{dt}\left(_tI^{1-\alpha}_b\left(e^{\gamma t}f'(^C_aD^{\alpha}_tx)\right)\right)=0$, that is, the quantity
$$_tI^{1-\alpha}_b\left(e^{\gamma t}f'(^C_aD^{\alpha}_tx)\right)$$
 following the classical approach, can be called a generalized fractional constant of motion.

\end{ex}

Now we generalize Theorem~\ref{Noether:without:time} to the case of several independent variables. For that consider the one-parameter family of transformations
\begin{equation}\label{tra:fam:m}
\bar{u}_j=h_j(t,x,u,s), \quad j=1,\ldots,m
\end{equation}
depending on a parameter $s$, $s\in (-\varepsilon,\varepsilon)$, where  $h_j$ are of class $C^2$ and such that $h_j(t,x,u,0)=u_j(t,x)$ for all $j\in\{1,\ldots,m\}$ and for all $(t,x,u)\in [a,b]\times \Omega \times \mathbb R^m$.  As before for $|s|\leq \varepsilon$ the linear approximation to transformation \eqref{tra:fam:m} is $\bar{u}_j\approx u_j+s\xi_j(t,x,u)$,
where $\xi_j(t,x,u)=\frac{\partial}{\partial s}h_j(t,x,u,s)|_{s=0}$. The following result generalizes Theorem~\ref{Noether:without:time}, and can be proved in a similar way (cf. \cite{ABM})

\begin{thm}\label{Noether:without:time:m}
Let the functional $z$, defined by differential equation \eqref{GVPS}, is invariant under the family of transformations \eqref{tra:fam:m}. Then
\begin{multline*}
\sum_{j=1}^n\left\{\mathcal{D}^{\alpha_j}\left[\lambda(t)\frac{\partial L}{\partial {^C_aD^{\alpha_j}_t}u_j}[u,z](t,x),\xi_j(t,x,u)\right]\right.\\
\left.+\sum_{i=1}^n\mathcal{D}^{\beta_{j,i}}\left[\lambda(t)\frac{\partial L}{\partial{^C_{a_i}D^{\beta_{i,j}}_{x_i}}u_j}[u,z](t,x),\xi_j(t,x,u)\right]\right\}=0,
\end{multline*}
where $\lambda(t):=\exp\left(-\int_a^t\int_\Omega  \frac{\partial L}{\partial z}[u,z](\tau,x) d^nx d\tau \right)$ and $\mathcal{D}^{\alpha}[f,g]:=f {^C_aD^{\alpha}_t}\cdot g- g\cdot{_tD^{\alpha}_b}f$, holds along the solutions of the generalized fractional Euler--Lagrange equations \eqref{EL:md}.
\end{thm}


\section*{Acknowledgements}

Ricardo Almeida is supported by Portuguese funds through the CIDMA - Center for Research and Development in Mathematics and Applications, and the Portuguese Foundation for Science and Technology ("FCT–Fundação para a Ciência e a Tecnologia"), within project PEst-OE/MAT/UI4106/2014. Agnieszka B. Malinowska is supported by the Bialystok University of Technology grant S/WI/02/2011. We thank the anonymous reviewer for his careful reading of our manuscript and his many insightful comments and suggestions.


\medskip
Received ; revised .
\medskip

\end{document}